\documentclass[onefignum,onetabnum]{siamart220329}

\usepackage{amsmath,amssymb,mathtools}
\usepackage[dvipsnames]{xcolor}
\usepackage[utf8]{inputenc}
\usepackage{booktabs}
\usepackage{enumitem}
\usepackage[margin=1in]{geometry}
\usepackage[colorlinks=true,linkcolor=Maroon,citecolor=Maroon,urlcolor=Maroon]{hyperref}
\usepackage{subcaption}
\usepackage{algorithm}
\usepackage{algpseudocode}
\newcommand*{\horzbar}{\rule[.5ex]{2.5ex}{0.5pt}}

\newcommand{\mcD}{\mathcal D}
\newcommand{\bz}{{\boldsymbol z}}
\newcommand{\by}{{\boldsymbol  y}}
\newcommand{\bx}{{\boldsymbol x}}
\newcommand{\bu}{{\boldsymbol u}}
\newcommand{\be}{{\boldsymbol e}}
\newcommand{\query}{{\boldsymbol q}}          
\newcommand{\val}{{\boldsymbol v}}             
\newcommand{\assign}{{\xi}}            
\newcommand{\assignv}{{\boldsymbol \xi}}            
\newcommand{\querytrue}{\boldsymbol{q}^*}    
\newcommand{\valtrue}{\boldsymbol{v}^*}       
\newcommand{\assigntrue}{\xi^*}      
\newcommand{\assigntruev}{\boldsymbol{\xi}^*}      
\newcommand{\querymap}{\hat{\boldsymbol q}}
\newcommand{\valmap}{\hat{\boldsymbol v}}

\newcommand{\assignext}{\hat{\xi}}                         
\newcommand{\assignextv}{{\hat{\boldsymbol \xi}}}                      

\newcommand{\reals}{\mathbb R}
\newcommand{\bbP}{\mathbb P}
\newcommand{\E}{\mathbb E}

\newtheorem{thm}{Theorem}[section]

\newtheorem{lem}[thm]{Lemma}

\newtheorem{defn}[thm]{Definition}

\title[EM-based iterations]{EM-based iterations for multiple instance learning on a query-value model}

\author[E. Levien]{Ethan Levien}
\address{Department of Mathematics, Dartmouth College}
\email{ethan.levien@dartmouth.edu}

\date{}

\begin{document}

\begin{abstract}
In multiple instance regression (MIR)  data are organized into bags (collections of instances in feature space) and the goal is to learn a mapping that assigns labels to bags. A typical assumption is that there is a so-called concept point in feature space, the proximity to which dictates the bag label. Motivated by modern MIR architectures which are based on attention, we study a softmax model that decouples the concept point and the labeling scheme. The two are respectively determined by a query direction $\query$ and a value direction $\val$ in feature space. This problem isolates a basic challenge of learning both $\query$ and $\val$ from bag-level supervision. From this model we derive a parametric family of iterations in the noiseless limit, which generalizes a method known as the EM-DD algorithm. 
We then derive concentration results for the MLE estimators of the query and value vectors obtained from a random selection of instances. Our result for the value vector shows that a single random initialization of the value vector already points in the correct direction on average, so that a polynomial (in the number of instances per bag and the feature dimension) number of bags is enough for the EM algorithm to converge in $O(1)$ steps with high probability. A key aspect of this analysis is the interplay between concentration of empirical covariance matrices and extremal statistics arising from the selection rule.
\end{abstract}
\maketitle

\section{Introduction}

Multiple Instance Learning (MIL) was introduced for chemical identification problems in \cite{dietterich1997solving} and has since become a major paradigm in machine learning \cite{Carbonneau2018survey}. In contrast to standard supervised learning where individual feature vectors are paired with labels, MIL training data consists of sets of feature vectors, called bags, together with a single bag-level label for each set; see Figure \ref{fig:1}. We represent each bag as a matrix $X^k$ with rows $\bx_i^k$ (the \emph{instances}) and bag labels $y_k \in \reals$ (or $\{0,1\}$). The goal in MIL is then to learn a pooling function $h:\cup_{n'} \reals^{n' \times d} \to \reals$ from the data $(X^k,y_k)$ which is assumed to be invariant to row permutations of $X^k$ \cite{dooly2002multiple,doran2016multiple,uriot2019learning,Uriot2019KernelEmbeddingMIR,guptalearning}.

The MIL framework applies across a wide range of applications. In drug design, one measures the binding affinity of a molecule to a target, which is often determined by a single high-binding conformation among the many conformations a molecule can adopt. Each conformation may be embedded as a feature vector and treated as an instance, with the molecule as the bag \cite{dietterich1997solving,jain1994compass,sliwoski2014computational,RayPage2001MIR}. In immune repertoire classification instances are embeddings of antibody sequences within an individual's repertoire and the bag label measures immune response. Typically immunity is determined by only a small fraction of disease-associated sequences \cite{widrich2020modern}. In image analysis, images serve as bags and patches or subregions as instances \cite{maron1998learning,zhou2006miml}.

 \begin{figure}[t!]
\centering
\includegraphics[width=0.9\textwidth]{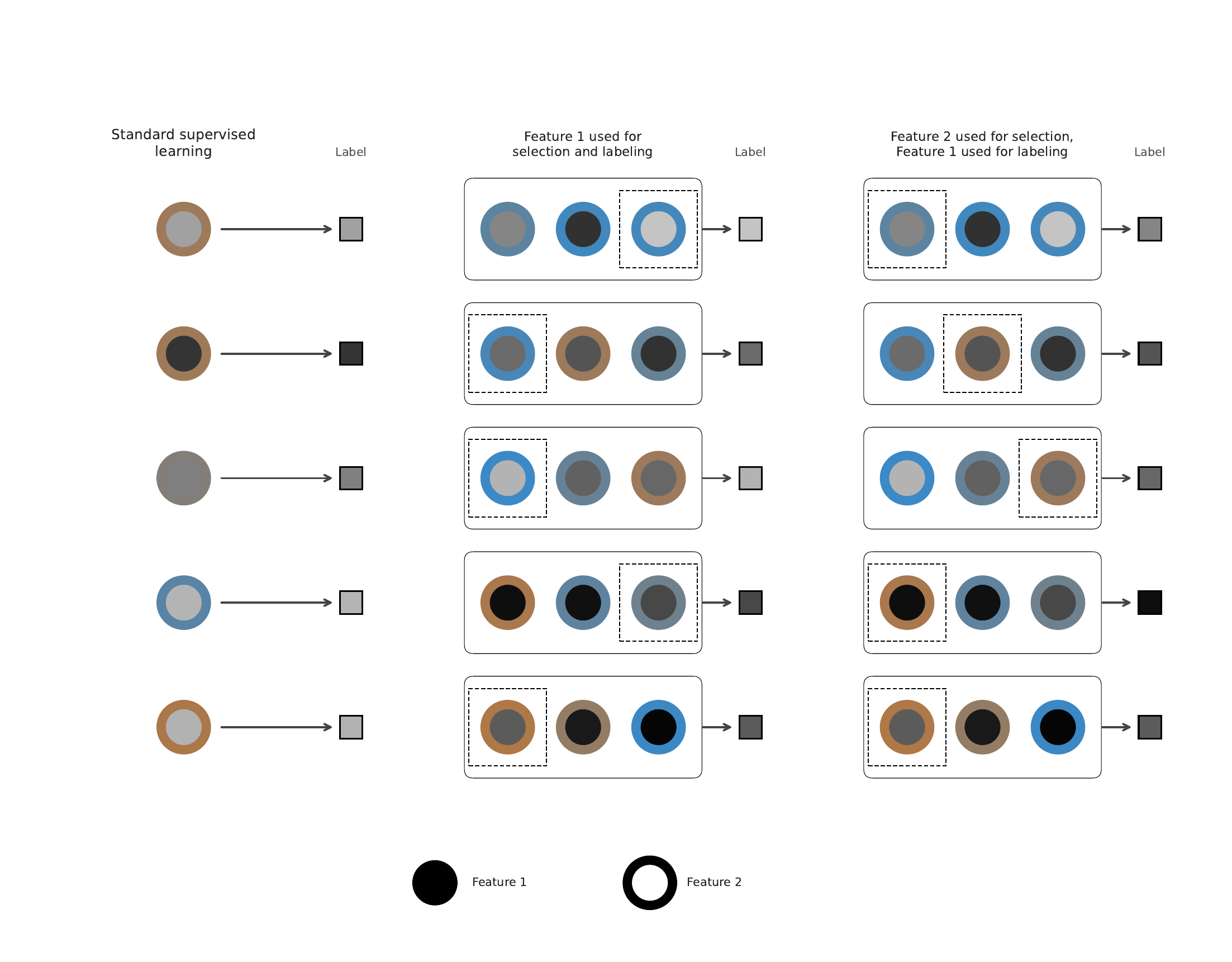}
\caption{Comparison of standard supervised learning and MIL.  Each row represents a bag of instances (circles). Instances live in feature space, which here is $\reals^2$ with coordinates $(x_1,x_2)$ represented by a hue (edge) and lightness (filling). In standard supervised learning each bag contains a single instance. In MIL, we select an active instance via some selection rule, which takes the lightest circle (middle) or the most orange circle (right).  The middle represents an MIL dataset where the feature used for selecting the active instance  is also used to label that instance.  In contrast, on the right figure the hue is used to select the active instance, but the label comes from the lightness feature. In our model, the directions in feature space used for selection and labeling will be denoted by $\query$ and $\val$ as a nod to query and value matrices in attention.   }\label{fig:1}
\end{figure}

The subject of this paper is the asymptotics of the MIL problem when the bag labels are determined by an \emph{active instance} which is selected via some extremal rule. Specifically, we study the pooling function 
\begin{equation}\label{eq:hardmax}
    h_{\infty}(X; \query, \val) = \bx_{\,\operatorname{argmax}_{1 \le j \le n}\, \bx_j \cdot \query} \cdot \val.
\end{equation}
which is parameterized by a \emph{query} $\query \in S^{d-1}$ and \emph{value} vector $\val \in \reals^d$. Following \cite{RayPage2001MIR},  we consider the regression setting, henceforth called Multiple Instance Regression (MIR), in which $\query$ and $\val$ are to be learned from labeled bags $\{(X^k, h(X^k; \query, \val))\}_{k \in [m]}$.   The notion that there is an extremal principle for determining which instance is used to label a bag dates to early work on MIL for drug design \cite{dietterich1997solving,maron1998learning}.
In particular, many models were based on the idea of a \emph{concept point} $\bx^*$ such that the bag label was determined by proximity to $\bx^*$, e.g., $h(X^k) = V(\min_i\, d(\bx_i, \bx^*))$ for some distance $d$ and potential $V$. Eq. \eqref{eq:hardmax} falls within a more general class of models where the features used to select the active instance (determined by $\query$) may differ from those used to assign the label to that instance (determined by $\val$); see Figure~\ref{fig:1} (middle and right panels).

Eq. \eqref{eq:hardmax} can be viewed as a simple case of an  Attention layer, where the pooling function is a softmax \cite{ilse2018attention}. Such models have become common in MIL and are sometimes called Hopfield pooling layers  \cite{widrich2020modern,ramsauer2021hopfield}. 
In attention based models instances are aggregated in a bag level embedding $\bz_k$ obtained as a weighted sum of instances: 
\begin{equation}
\bz_k = \sum_{j=1}^na_{k,j}\bx^k_j,\quad a_{k,j} = \frac{e^{\lambda s(\bx^k_j)}}{\sum_{j'}e^{\lambda s(\bx^k_{j'})}}
\end{equation}
for some scoring function $s(\bx^k_j)$. The scaling factor, or inverse temperature, $\lambda \in \reals$ is a means of controlling the selection strength on the active instance. The aggregated point in feature space $\bz_k$ is then used to determine the bag label. 
Similar to the active instances $\xi_k$, $\bz_k$ acts as a form of data compression in the following sense. Conditioning on $\bz_k$ removes the dependence on the bag, while conditioning on $\xi_k$ removes the dependence on all the instances except $\bx^k_{\xi_k}$. These two forms of data compression are related if we treat $a_{k}$ as the bag-dependent prior distribution of $\xi_k$: 
\begin{equation}
a_{k,j} = \bbP(\xi_k = j|\bx^k_1,\dots,\bx^k_n)  \implies \bz_k = \E_{\xi_k \sim a_k}[\bx_{\xi_k}^k]. 
\end{equation} 
This latent variable interpretation has also been explored in the context of self-attention and transformer models, and allows Expectation-Maximization (EM) approaches to be applied to attention layers \cite{ilse2018attention}, which are typically trained with gradient descent  \cite{luo2020weakly,dooly2002multiple}.   The pooling function \eqref{eq:hardmax} can be viewed as the large $\lambda$ limit of an attention model with linear scoring, $s(\bx) = \query \cdot \bx$ followed by a linear regression on the aggregated statistic $\bz$. 
That is, Eq. \eqref{eq:hardmax} is the limit of the pooling function
\begin{equation}\label{eq:softmax}
  h_{\lambda}(X;\query,\val) \coloneqq \sum_{j=1}^n \frac{e^{\lambda\, \bx_j\cdot\query}}{\sum_{j'}e^{\lambda\,\bx_{j'}\cdot\query}}\,(\bx_j\cdot\val).
\end{equation}

Prior to the introduction of attention models in  MIL, the  EM algorithm played a major role. Specifically, a variant known as EM-DD  \cite{zhang2001emdd} combined EM with the Diverse Density (DD) approach of Maron and Lozano-P{\'e}rez \cite{maron1998framework}.
The EM algorithm is an alternating optimization procedure for Maximum-Likelihood Estimation (MLE) which is based on the idea that once a latent variable is conditioned on, the MLE becomes tractable \cite{amari2016information,hino2024geometry}. Conditioning on the indices $\xi_1,\dots,\xi_m$ of  active instances in each bag turns the problem into a standard supervised learning problem, and they are therefore  a natural latent variable.    Beyond diverse density, classical MIL methods also include SVM and kernel-based approaches \cite{Gartner2002MIKernels,Andrews2002MISVM}.
Despite these developments, Attention based models are typically trained with gradient descent, with $\lambda$ being a learned parameter \cite{widrich2020modern}. While EM has been applied to attention layers in other contexts \cite{li2019expectation}, there does not appear to be any rigorous understanding of how it performs on attention layers for simple models, such as our query-value model.

We treat Eq. \eqref{eq:hardmax} as a limit of a probabilistic model, which is equivalent to the attention model given by Eq. \eqref{eq:softmax}. In particular, we take noise in both selection (which instance is active) and labeling (how $y$ is determined from the active instances) to vanish together.  That is,  $p(\xi|X)$ and $p(y|X,\xi)$, which respectively control the noise in the selection of the active instance and the labeling of the bag given this instance, are approaching delta functions. This leads to a general family of assignment iterations arising from the EM algorithm. These iterations treat the active instance $\assigntrue_k = \operatorname{argmax}_{1 \le j \le n}\, \bx_j^k \cdot \querytrue$ as a latent variable. In particular, at each step of the iteration the value vector is estimated from a linear regression problem where the rows of the design matrix are $\bx^k_{\assign_k}$, with $\assign$ being a current estimate of the assignment. If only a fraction $f$ of the bags have the correctly assigned instance -- that is $f = m^{-1}\sum_k 1\{\assign_k = \assigntrue_k\}$ -- then a fraction $1-f$ of the rows will not have been used to generate the labels, while the other $f$ will. Moreover, those $f$ rows will be biased towards the query direction. Understanding the resulting linear regression problem therefore involves the interplay of extremal statistics and concentration for the relevant covariance matrices. We will analyze this regression problem after deriving the EM algorithm for the query-value model and presenting some numerical results.

 \subsection{Organization and contributions of this paper}
The paper is organized as follows. 

  \begin{itemize}
  \item In Section~\ref{sec:em} we derive a family of iterations, called ${\rm EM}_{\kappa}$ iterations, as a limit of the EM algorithm for a soft query--value model. The iterations are parameterized by $\kappa$, which controls the relative roles of the query and value vectors in determining the next assignment.  We show that
  the EM-DD method, when linearized, is a variant of the $\kappa=1$ member of this family (which we refer to as the $\widetilde{\rm EM}_1$ iteration). This therefore clarifies the relationship between previous EM based approaches applied to concept point models, and the general EM algorithm for the query-value model. 
  \item Section~\ref{sec:numerics} presents numerical examples on synthetic data with Gaussian instances. We find that the linearized EM-DD iteration ($\widetilde{\rm EM}_1$) converges in a few steps to the correct assignment when the query and value directions are
  aligned, but fails to do so when they are misaligned. This is to be expected since the algorithm is misspecified in this case, but interestingly, it converges to a partially correct assignment faster and more reliably than the correctly specified EM iteration. 
Motivated by this, we find that a staged method converges faster than the constant-$\kappa$ method. In this method, a $\kappa$ schedule alternating between $0$ and $1$ is run for a fixed number of steps, followed by a fixed schedule.
  \item Section~\ref{sec:manybag} contains theoretical results for Gaussian instances. Specifically, we state concentration theorems for the value and query vectors obtained from a random assignment, conditionally on a given fraction of correctly matched instances. Using our concentration result for the value vector we show that a \emph{single} uniformly random initial assignment already points, on average, in the correct direction for $\widetilde{\rm EM}_1$; consequently, once $m \gtrsim dn^2(\ln n)^6$ bags are available, this one assignment is likely to yield the true value vector in a single step. 
  \end{itemize}
  
  \subsection{Related work}

\subsubsection{CEM and clustering}
The EM-DD algorithm is closely related to Classification EM (CEM). Here, the likelihood is averaged over the data posterior, with a ``hardened'' assignment step \cite{celeux1992classification}. CEM was proposed for clustering because it allows many clustering algorithms to be viewed as variants of the EM algorithm \cite{zhang2009clustering}.
However, while clustering and MIL are both solved with EM, they differ in how the data compression is used.
In the regression setting, once $\vartheta$ is estimated, the assignment is determined by the observed pairs $(y_k,X^k)$. In other words, we do not need to integrate information across bags at the E step. In clustering the latent variable is the cluster assignment and to obtain the assignment in the clustering case, we need to combine information across points, but once this is done the parameter estimates are typically determined within each cluster, or are trivial to compute given the cluster assignment (for the prior cluster probabilities, for example \cite{dempster1977maximum,mclachlan2008algorithm}).

This contrast also has a direct bearing on initialization. In the clustering setting, a random partition of the data produces cluster means that all converge to the global population mean as $m \to \infty$, so a random initialization carries no information about the true cluster structure; signal only emerges once the E-step has incorporated structure from the data. In the latent regression setting, we can in some cases learn from a completely random assignment (see \cite{guptalearning}).

\subsubsection{Other extremal regression problems}
There are also some related models worth noting. Another problem involving max non-linearities is the max-linear factor model, where the goal is to learn $\val$ from $y_k = \max_i\{X_{k,i}\beta_i\}$ \cite{tran2022tropical}. In that case the problem can be viewed as linear regression on the tropical semiring and solved with a tropical matrix factorization. This contrasts with Eq.~\eqref{eq:hardmax}, where the matrix--vector product $X^k\val$ occurs in Euclidean space and the max is taken afterward, making algebraic approaches more challenging to develop.

\subsubsection{MIR with Gaussian instances}
Prior works \cite{Chauhan2024GenLearnMIR,guptalearning} have studied the asymptotics of a regression problem with Gaussian instances when the labeling mechanism is linear. Some of their results are similar in spirit to our Theorem \ref{thm:conc}. In particular Lemma 3.1 in \cite{guptalearning} provides bounds on the probability of retrieving the correct regression coefficient from randomly selected bags, with each instance being assigned the same bag label. We discuss the connection in detail in Section \ref{sec:manybag}. 

\subsubsection{Statistical guarantees for the EM algorithm}

There is also a substantial literature on statistical guarantees for EM-type algorithms in latent-variable models, including finite-sample and population-level analyses of EM \cite{balakrishnan2017statistical,xu1996convergence}. Our setting differs in that the noiseless extremal limit leads to hardened assignment iterations whose behavior is governed by empirical covariance concentration and Gaussian extreme-value statistics.

 \section{EM-based assignment iterations}\label{sec:em}
 
  \subsection{Notation}
  We use the notation $[n]  = \{1,\dots,n\}$ for any $n \in {\mathbb N}$. $\|\cdot\|$ will denote the $2$-norm for vectors and the spectral norm for matrices. We use $\theta_{\bx,\by}$ to denote the angle between vectors $\bx$ and $\by$. 
  As described in the introduction, our predictors are organized in $m$ bags each containing $n$ instances. Bags are represented as $n \times d$ design matrices $X^k$ whose rows are $d$-vectors $\bx^k_j$.  A vector $\assignv = (\assign_1,\dots,\assign_m) \in [n]^m$ is called a \emph{bag assignment} and assigns each bag an instance.  For an assignment $\assignv$ and data $\mcD = \{(X^k, h(X^k; \query, \val))\}_{k \in [m]}$ we define a matrix $X(\assignv) \in \reals^{m \times d}$ whose rows are the assigned instances from each bag; that is,
\begin{equation}
X(\assignv) = \begin{bmatrix}
\horzbar&  \bx_{\assign_1}^1 &\horzbar   \\
&   \vdots  & \\
\horzbar& \bx_{\assign_m}^m   &\horzbar
\end{bmatrix} \in \reals^{m \times d}.
\end{equation}
For two assignments $\assignv$ and $\assigntruev$, we denote the Hamming distance by $d_H(\assignv,\assigntruev) \coloneqq \#\{k  \in [m]:\assign_k \ne \assigntrue_k\}$. For two assignments $\assignv, \assignv' \in [n]^m$, we denote the cross covariance matrix of assigned instances by 
\begin{equation}
\hat{\Sigma}(\assignv,\assignv') = \frac{1}{m}X(\assignv)^TX(\assignv'). 
\end{equation}

\subsection{The EM algorithm for a general latent variable regression model}

Here we provide a brief introduction to the EM algorithm as it applies to a generic regression problem with a latent variable $\assign$ and parameters $\vartheta$. For the MIL problem, $\assign$ will be the assignment defined above and $\vartheta$ will be the value and query vectors; however, the formulation is able to accommodate a generic regression problem with a latent variable. 

The starting point is the likelihood $p(y|X,\vartheta)$ of the response variable given the predictors and parameters $\vartheta$.
\begin{equation}
L(\vartheta) =  \frac{1}{m} \sum_{k=1}^mL_k(\vartheta) = \frac{1}{m} \sum_{k=1}^m\ln p(y_k|X^k,\vartheta)
\end{equation}
where the likelihood of a single datum is obtained by marginalizing over $\assign$: $p(y|X,\vartheta) = \sum_{\assign} p(y,\assign|X,\vartheta)$.
We assume the model is well specified, meaning the data $(y_k,X^k)$ is sampled from $p(y_k|X^k,\vartheta^*)\pi(X^k)$ for the true parameters $\vartheta^*$ and predictor prior $\pi(X^k)$. The idea of the EM algorithm is that the likelihood $p(y|X,\assign,\vartheta)$ obtained by conditioning on the latent variable $\assign$ renders the likelihood more tractable (typically convex).

By Jensen's inequality, $L(\vartheta)$ is bounded below by
\begin{equation}\label{eq:Lthetardef}
L(\vartheta,r) = \frac{1}{m} \sum_{k=1}^m\E_{\assign \sim r_k}\left[\ln  \frac{p(y,\assign|X,\vartheta)}{r_k(\assign)}\right]
\end{equation}
where $r= m\otimes_k^m r_k$. $L(\vartheta,r)$ is referred to as the \emph{variational lower bound} (VLB) and satisfies
\begin{align}
L(\vartheta) - L(\vartheta,r) &= \frac{1}{m} \sum_{k=1}^m\E_{\assign \sim r_k}\left[\ln \frac{r_k(\assign)}{p(\assign|y_k,X^k,\vartheta)}   \right] \\
&=\frac{1}{m} \sum_{k=1}^mD_{\rm KL}(r_k,p(\cdot|y_k,X^k,\vartheta)) \ge0 \label{eq:KLbound}
\end{align}
Finding the distribution $r$ which maximizes the VLB is therefore equivalent to minimizing the expected KL divergence between the latent variable posterior and the distribution $r_k$, yielding  $\hat{p}_k(\assign) = p(\assign|y_k,X^k,\hat{\vartheta})$.

The EM algorithm is the alternating iteration on the VLB
\begin{align}\label{eq:EM}
r^{(t+1)} &= \operatorname{argmax}_{r}L(\vartheta^{(t)},r)\\
\vartheta^{(t+1)} &= \operatorname{argmax}_{\vartheta}L(\vartheta,r^{(t+1)})
\end{align}
where $\vartheta^{(0)}$ is typically sampled uniformly, but may also be drawn from some prior distribution.
Computing $L(\vartheta,r^{(t+1)})$ is referred to as the expectation step, while the optimization to find $\vartheta$ is the maximization step.  Usually this is written in the equivalent form
\begin{equation}\label{eq:EMQ}
\vartheta^{(t+1)} = \operatorname{argmax}_{\vartheta} \,\,Q(\vartheta;\vartheta^{(t)})
\end{equation}
where $Q$ is the so-called $Q$-function given by
\begin{equation}\label{eq:EMQdef}
Q(\vartheta;\vartheta^{(t)}) \coloneqq \frac{1}{m} \sum_{k=1}^m \sum_{\assign_k}p(\assign_k|y_k,X^k,\vartheta^{(t)}) \ln p(y_k,\assign_k|X^k,\vartheta)
\end{equation}
It can be checked that Eq.~\eqref{eq:EMQdef} is equivalent to \eqref{eq:EM}.
The basic theory of EM presented in \cite{mclachlan2008algorithm} establishes that the likelihood is non-decreasing along the sequence $\{\vartheta^{(t)}\}_{t \ge0}$.

In the case of MIL, a common assumption is that the assignment is generated by an extremal principle, thus $p(\assign_k = j|\vartheta)$ is a $\delta$ function. Computing the extremal assignment at each step gives rise to 
\begin{align}\label{eq:CEM}
\assign_k^{(t+1)} &= \operatorname{argmax}_{\assign}L_k(\vartheta^{(t)},\assign_k),\quad k = 1,\dots,m\\
\vartheta^{(t+1)} &= \operatorname{argmax}_{\vartheta}L(\vartheta,\assignv^{(t+1)})
\end{align}
where $L(\vartheta,\assignv)$ is shorthand for $L(\vartheta,\delta_{\assignv})$ with $L$ defined by Eq. \ref{eq:Lthetardef} (and similarly for $L(\vartheta,\delta_{\assignv})$). 
This ``hard'' form of the EM is referred to as the CEM in the context of clustering  \cite{celeux1992classification} and is related to EM-DD \cite{zhang2001emdd,zhou2006miml} algorithm for MIL, as we explain below.  If Eq. \eqref{eq:CEM} is to be understood as the limit of the EM when the posterior $\hat{p}_k(\assign)$ approaches a delta function, we should also be taking \emph{both} the likelihood $p(y_k|\assign_k,X^k,\vartheta)$ and the assignment prior $p(\assign_k|X^k,\vartheta)$ to approach delta function simultaneously.  Depending on the relative rates at which these approach delta functions, we will arrive at different updates for the assignment $\assign_k$. However, the data generating process will not depend on this choice.

 \subsection{Query-value model and $Q$-function}
 We now introduce a probabilistic model for which the extremal regression problem described in the introduction is a limit. The learnable parameters are the query--value pair $\vartheta = (\val,\query) \in \reals^d \times S^{d-1}$ and the latent variable is the assignment $\assignv$. For each $k$, we introduce a distribution on $[n]$ depending on the instances: 
 \begin{equation}
 w_{k,j}(\query) = e^{\query \cdot \bx^k_{j} - \psi_k(\query)},\quad \psi_k(\query) =  \ln \sum_{j'}e^{\query \cdot \bx^k_{j'}}\\
 \end{equation}
 This can also be seen as an exponential family with natural parameter $\query$ and uniform prior over the bag instances  \cite{amari2016information}. 
 The function $\psi_k$ is known as the cumulant generating function.  
Our data generating model may be written compactly as 
\begin{subequations}\label{eq:em_model}
  \begin{align}
\bbP\left( \assign_k=j|X^k,\query\right) &=  w_{k,j}(\lambda\query) \label{eq:em_model_a}\\
y_k &= \val \cdot \bx^k_{\assign_k} + \epsilon,\quad \epsilon \sim {\rm Normal}(0,\sigma_{\epsilon}^2)
\label{eq:em_model_b}
 \end{align}
\end{subequations}
For this model, the iteration given in Eq. \eqref{eq:CEM} has a simple form, which is amenable to analytical study using tools from high dimensional probability theory and extreme statistics.

Notice that conditioning on the assignment removes the dependence of the likelihood on the $\query$ and turns the problem of estimating $\val$ into a linear regression problem; that is
\begin{equation}
p(y_k|X^k,\assign_k,\val,\query) = p(y_k|X^k,\assign_k,\val)=  \frac{1}{\sigma_{\epsilon} \sqrt{2\pi}}e^{-(y_k-  \bx_{\assign_k}^k \cdot \val)^2/2\sigma_{\epsilon} ^2}.
\end{equation}
In this way, $\assign$ acts as a compression of the data, which is precisely what we want for the latent variable in the EM algorithm. The marginal likelihood 
\begin{equation}
p(y_k|X^k,\val,\query)= \sum_{j}w_{k,j}(\query) p(y_k|X^k,\assign_k=j,\val,\query) 
\end{equation}
is related to the \emph{diverse density} described in \cite{maron1998framework,zhang2001emdd}. A key difference, however, between our model and theirs is the query--value geometry. Specifically, in their model the bag label is determined by the distance to the ``concept point,'' which in our case is the query.

The posterior probabilities of assignments will be written as 
\begin{equation}
\hat{p}_{k,j}(\val,\query) \coloneqq p(\assign_k=j \mid y_k,X^k,\val,\query)
\end{equation}
which are given by 
\begin{align}
\hat{p}_{k,j}(\val,\query)
&= \frac{p(y_k\mid \assign_k=j,X^k,\val)w_{k,j}(\lambda \query)}{\sum_{j'}p(y_k\mid \assign_k=j',X^k,\val)w_{k,j'}(\lambda \query)} = \frac{\exp\!\left(-\frac{(y_k - \val \cdot \bx^k_j)^2}{2 \sigma_{\epsilon}^2} + \lambda \query \cdot \bx^k_{j}\right)}{\sum_{j'} \exp\!\left(-\frac{(y_k - \val \cdot \bx^k_{j'})^2}{2 \sigma_{\epsilon}^2} + \lambda \query \cdot \bx^k_{j'}\right)}.
\end{align}
and complete-data log-likelihood for bag $k$ is
\begin{equation}\label{eq:pykassignkj}
\ln p(y_k,\assign_k=j\mid X^k,\val,\query)
= -\frac{(y_k-\val \cdot \bx_j^k)^2}{2\sigma_{\epsilon}^2}
+ \lambda\, \query \cdot \bx_j^k
- \psi_k(\lambda \query)
- \ln(\sigma_{\epsilon}\sqrt{2\pi}).
\end{equation}

Combining all this, we arrive at the $Q$-function of the EM algorithm, which can be written as linear combination of terms corresponding to the value and query updates: 
\begin{align}
Q(\val,\query;\hat{p}(\val^{(t)},\query^{(t)})) &= Q_1(\val;\hat{p}(\val^{(t)},\query^{(t)})) + Q_2(\query;\hat{p}(\val^{(t)},\query^{(t)}))
\end{align}
with
\begin{align}
Q_1(\val;p)
&= -\frac{1}{2m\sigma_{\epsilon}^2}\sum_{k=1}^m\sum_{j=1}^np_{k,j}(y_k-\val \cdot \bx_j^k)^2, \\
Q_2(\query;p)
&= \frac{1}{m}\sum_{k=1}^m \left[\lambda \query \cdot  \sum_{j=1}^n p_{k,j}\bx_j^k -\psi_k(\lambda \query)\right].
\end{align} 
These functions are to be maximized independently to obtain the query and value updates.
The value update is straightforward to solve and has the form of a weighted least-squares problem, and therefore has a closed-form solution. The query update is also a convex optimization problem when $\hat{p}_{k,j}> 0$, since $\psi_k$ is convex, however it does not have a closed form solution. In the next subsection we introduce an approximation in the hardened limit.

\subsection{Assignment iterations for the hardened limit}
The extremal regression model introduced in the introduction can be understood as the limit of Eq. \eqref{eq:em_model_b} as $\sigma_{\epsilon}^2,\lambda^{-1} \to 0$. Specifically, to obtain Eq. \eqref{eq:hardmax}, we define  
\begin{equation}
\kappa \coloneqq \frac{\lambda}{\tau},\quad \tau \coloneqq  \lambda + \frac{1}{2 \sigma_{\epsilon}^2}.
\end{equation}
and take $\tau \to \infty$ while keeping $\kappa$ fixed. 
This leads to the hard assignment update
\begin{equation}\label{eq:assignkappa}
 \assignext_k^{\kappa}(\query,\val) = \operatorname{argmax}_{j}\left[\kappa\,\bx_j^k\cdot\query- (1-\kappa)(y_k-\bx_j^k\cdot\val )^2\right].
\end{equation}
We denote the full assignment mapping $\assignextv^{\kappa}: \reals^d \times \reals^d \to [n]^m$ by
\begin{equation}
\assignextv^{\kappa}(\query,\val) \coloneqq (\assignext_1^{\kappa}(\query,\val),\dots,\assignext_m^{\kappa}(\query,\val)). 
\end{equation}
This tells us how to compute the expectation (E-step) in the hard limit of the EM algorithm, since in the $\tau \to \infty$ limit the expectation is simply given by the index $\assignext_k^{\kappa}(\query,\val)$ term in the sum for each bag $k \in [m]$. 

Now let $\querymap(\assign)$ and $\valmap(\assign)$ be the MLE for the query and values given an assignment $\assign$. That is, these are obtained by maximizing $Q_1$ and $Q_2$ respectively. The precise estimates we use are given in the next subsection.   An iteration on assignment space is given by
\begin{equation}\label{eq:kappaEM}
\assignv^{(t+1)} =  \assignextv^{\kappa}(\hat{\query}(\assignv^{(t)}),\valmap(\assignv^{(t)}))
\end{equation}
Taking $\assignv^{(0)} \sim {\rm Uniform}([n]^m)$, this gives an algorithm for finding $\assignv$.  We will refer to this as the ${\rm EM}_{\kappa}$ iteration. 
In subsequent sections we investigate properties of this iteration when applied to Gaussian instances with well-specified labeling ( meaning the labels are derived from the model in the limit $\tau \to \infty$). Importantly, in this limit the data does not depend on $\kappa$. In other words, when $\lambda = O(1/\sigma_{\epsilon}^2)$ and $\lambda \to \infty$ the data ceases to depend on either parameter, yet the limiting behavior of the EM algorithm does. $\kappa$ should therefore be treated as a tuning parameter when dealing with such extremal data. 

For the boundary case $\kappa=0$, we write
\begin{equation}
\assignv^{(t+1)} = \hat{\assignv}^0(\valmap(\assignv^{(t)})). 
\end{equation}
 Technically, this limiting case of EM$_{\kappa}$ assumes that the labeling is far less noisy than the selection, thus there is no point in using the signal from the selection rule (which comes from $Q_2$) when we have perfect information from the labels. However, since the query MLE does not have a closed form, it may still be advantageous to use it in some situations where both the selection and labeling are noiseless. The query can also be computed after the iteration is run for some time. 

When $\query \propto \val$ we could use the value map as input for the query argument in the assignment map: 
\begin{equation}\label{eq:DDEM}
\assignv^{(t+1)} =  \assignextv^{\kappa}(\valmap(\assignv^{(t)}),\valmap(\assignv^{(t)}))
\end{equation}
This makes sense when the query and value are aligned, since there is no need to rely on the approximate query $\hat{\query}_a$. We denote Eq.~\eqref{eq:DDEM} by $\widetilde{\rm EM}_{\kappa}$ to distinguish it from the EM$_{\kappa}$ algorithms which are derived for any $\val$.  In this special case $\kappa=1$, $\widetilde{\rm EM}_1$ becomes a linearization of the EM-DD algorithm from \cite{zhang2001emdd}. In particular, the formulation in \cite{zhang2001emdd} uses a nonlinear labeling scheme where the bag label is $y_k = \max_j e^{-\gamma ||\val -\bx^k_j ||_2^2}$.  If $\gamma$ is small and simultaneously $d$ is large so $\bx^k_j$ are concentrated on the sphere (or the instances are normalized), the first order approximation is $e^{-\gamma||\val -\bx^k_j ||_2^2} \approx \val \cdot \bx_j^k + \text{constant terms}$.

\subsection{Query and value updates} 
 Putting the hard assignment into $Q_1$ yields the value mapping 
 \begin{equation}
 \valmap(\assignv) = {\rm argmin}_{\val \in \reals^d}\|\by - X(\assignv)\val\|_2^2. 
 \end{equation}
 In other words, $\valmap(\assign)$ is the solution to the regression problem with design matrix $X(\assignv)$. The maximization problem for the query does not have a closed form solution and there are different approaches we can take to update it. The gradient of the $Q_2$ function is 
\begin{equation}
\frac{1}{\lambda}\nabla Q_2(\query;p(\val^{(t)},\query^{(t)})) = \frac{1}{m}\sum_{k=1}^m \left[\sum_{j=1}^n p_{k,j}(\val^{(t)},\query^{(t)})\bx_j^k - \frac{1}{\lambda}\nabla \psi_k(\lambda \query) \right]
\end{equation}
where we have used $\lambda = \tau\kappa$.
Note that 
\begin{equation}
\frac{1}{\lambda} \nabla  \psi_k(\lambda \query)  = \frac{\sum_{j}\bx^k_je^{\lambda \bx^k_j \cdot \query}}{\sum_{j'}e^{\lambda \bx^k_{j'} \cdot \query}},
\end{equation}
which we recognize as softmax pooling of the $k$th bag, also known as the Modern Hopfield Update rule \cite{ramsauer2021hopfield}. 
When $\assignv = \assignv^{(t)}$, the limit $\tau \to \infty$ gives 
\begin{equation}
 Q_2(\query;p(\val^{(t)},\query^{(t)})) \to (1/\lambda)Q^*(\query,\assignv^{(t)}) \coloneqq \frac{1}{m}\sum_{k=1}^m\left[\query \cdot \bx_{\assign_k^{(t)}}^k - \max_{1 \le j \le n}\bx_j^k \cdot \query\right] 
\end{equation}
In general, $Q^*(\query,\assignv)$ does not have a unique maximizer on the unit sphere; the function is maximized on a cone where ${\assign_k} = \operatorname{argmax}_{1 \le j \le n}\bx_j^k \cdot \query$. However, as $n$ grows these cones converge to points 
which, for Gaussian instance, are well approximated by the normalized average over instances:
\begin{equation}\label{eq:qhat_a_def}
\hat{\query}_a(\assignv) \coloneqq  \frac{\bar{\bx}_{\assignv}}{\|\bar{\bx}_{\assignv}\|}, \qquad \bar{\bx}_{\assignv} \coloneqq \frac{1}{m}X(\assignv)^T{\bf 1}  = \frac{1}{m}\sum_{k=1}^m \bx_{\assign_k}^k.
\end{equation}
In the next section we will provide a justification for using $\hat{\query}_a(\assignv)$ as an approximation the optimization problem for $Q^*(\query,\assignv)$. Replacing $\querymap$ with $\hat{\query}_a$ in Eqns \eqref{eq:kappaEM} and \eqref{eq:DDEM} yields two families of iterations with a single parameter, $\kappa$.

\section{Examples on synthetic data}\label{sec:numerics}

\subsection{Experimental setup}
As proof of concept and to establish some basic properties of the various iterations, we perform simulations where the model is well-specified and the instance distribution is i.i.d.\ Gaussian. This means the true labels are generated by the extremal pooling function with no observation noise $(\sigma_{\epsilon}=0)$: 
\begin{equation}\label{eq:datassump}
\bx^k_{j} \underset{\text{i.i.d}}{\sim } {\rm Normal}(0,\Sigma),\quad y_k = h_{\infty}(X^k,\querytrue,\valtrue).
\end{equation}
We once again emphasize that $\kappa$ is absent from the data generating process in the limit $\tau \to \infty$, and is therefore treated as an algorithmic parameter.  It should be noted that if we knew our data was generated in this way, we could find $\val$ and $\query$ by searching all possible assignments for a small number of bags (say $2$) and performing regression on each one. This is however highly sensitive to measurement noise. More importantly, our goal in this paper is not to solve the MIL problem for this particular data, rather to use this limiting case to probe the properties of the EM algorithm where there is strong selection.

\begin{figure}[h!]
\centering
\includegraphics[width=0.8\textwidth]{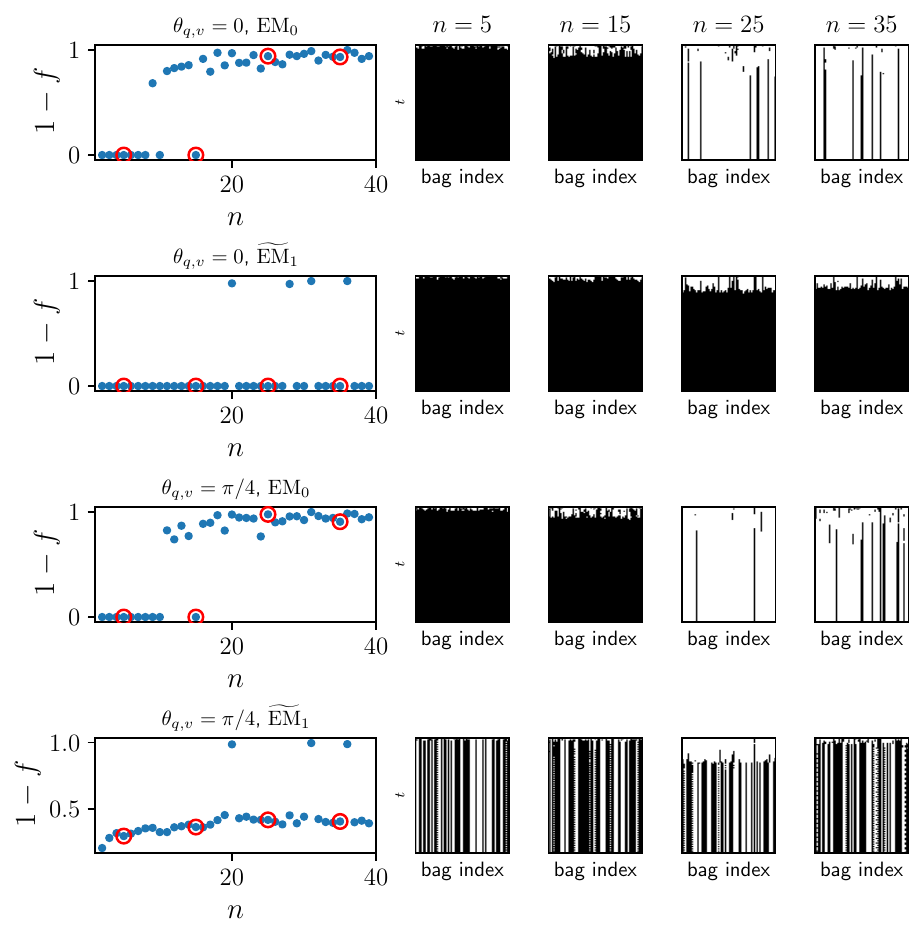}
\caption{Comparison of the ${\rm EM}_{0}$ iteration~\eqref{eq:kappaEM} (top row) and $\widetilde{\rm EM}_1$~\eqref{eq:DDEM} (bottom row) as a function of the number of instances per bag $n$ at $t=100$ steps from a uniformly random initial assignment. The image plots (right) show which bags have correctly identified active instances. For clarity, we only show a subset of the instances. The $n$ values corresponding to these examples are circled (left). Parameter values: $\theta_{\querytrue,\valtrue}=\pi/4$, $m=5000$, $d=15$, instances i.i.d.\ $\mathrm{Normal}(0,I_d)$.}
\label{fig:2}
\end{figure}

\begin{figure}[h!]
\centering
\includegraphics[width=0.9\textwidth]{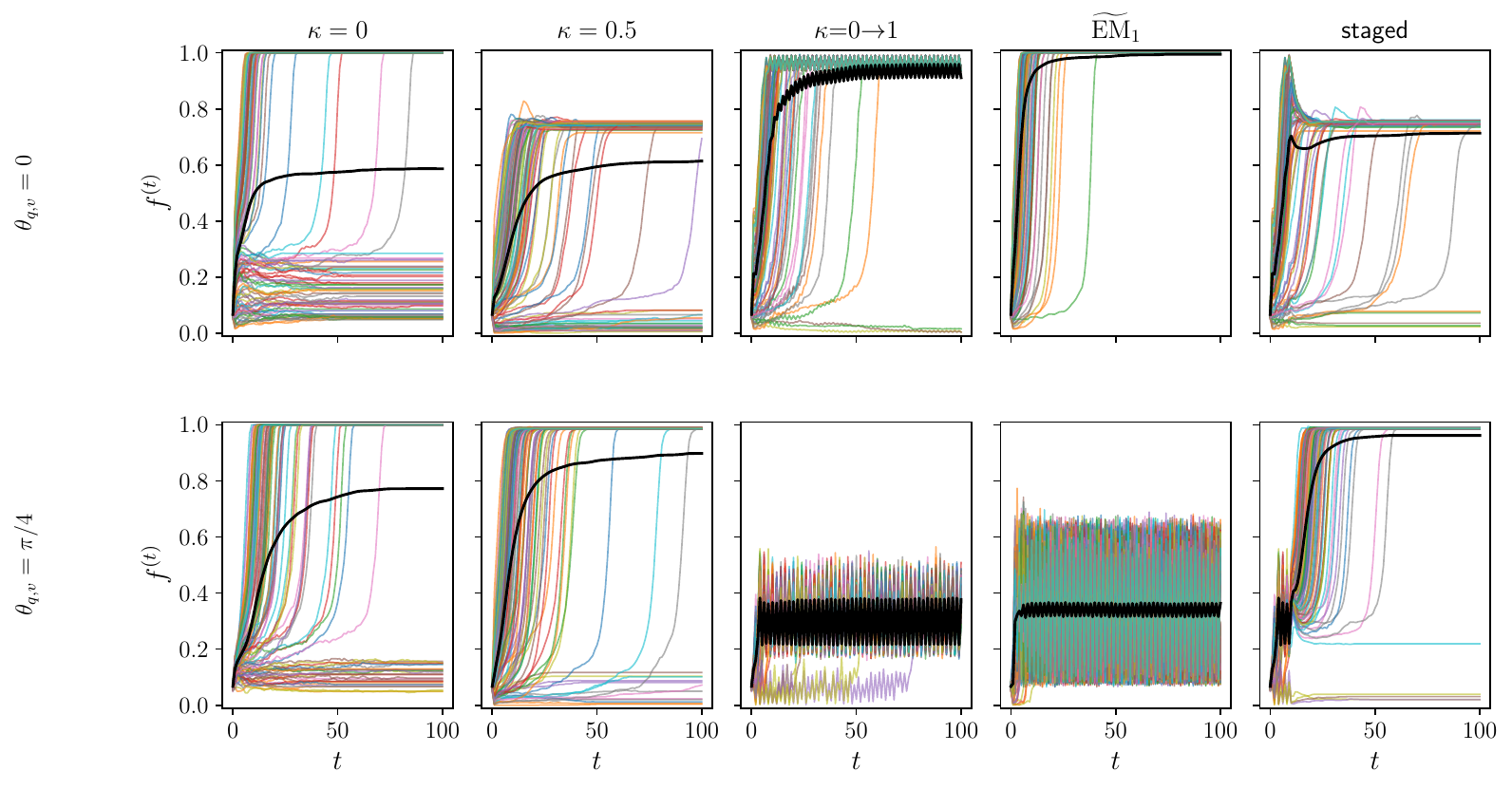}
\caption{Trajectories of the match fraction $f^{(t)}$ over iterations. Each panel shows a different variant of the hard EM iteration. Faint colored curves are individual trajectories initialized from random assignments. The bold black curve is the mean over all trajectories. Parameter values: $m=1000$, $n=50$, $d=10$, $\theta_{\querytrue,\valtrue}=0$.}
\label{fig:3}
\end{figure}

\begin{figure}[h!]
\centering
\includegraphics[width=0.9\textwidth]{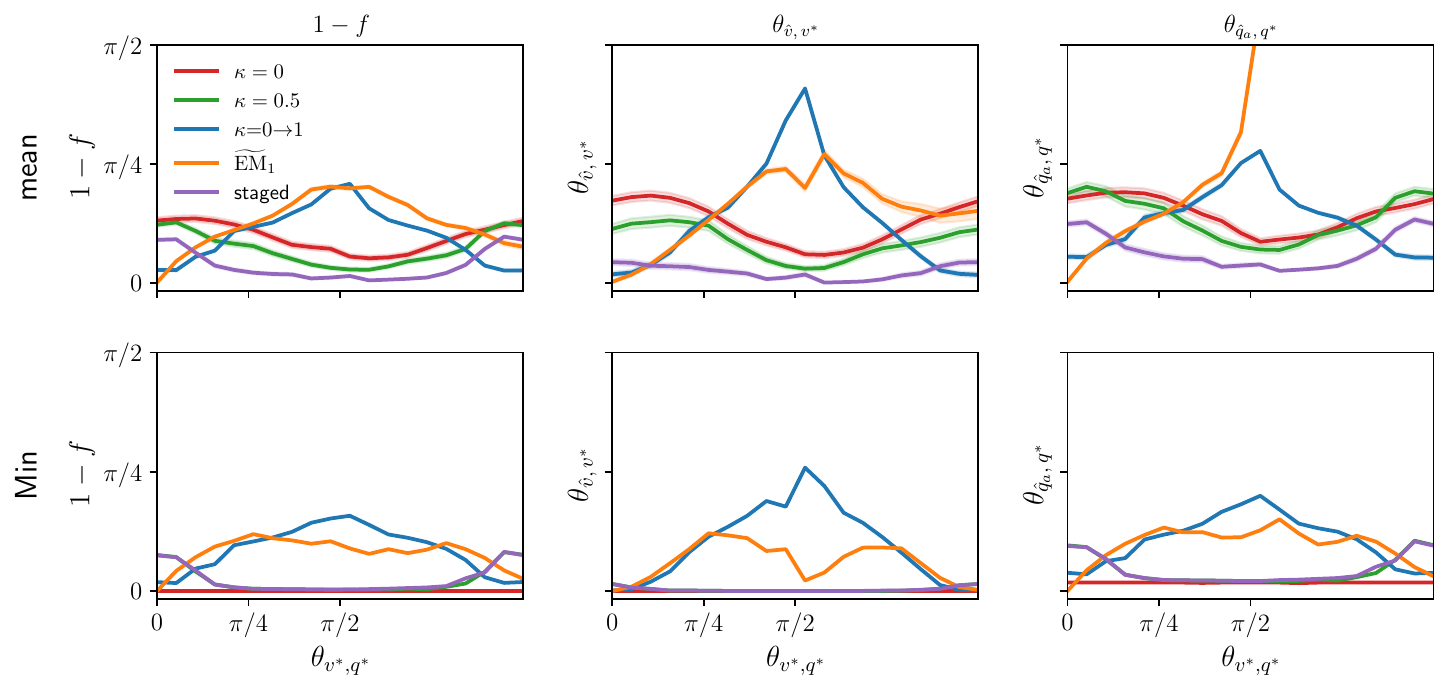}
\caption{Performance of the ${\rm EM}_{\kappa}$ schedules, $\widetilde{\rm EM}_1$, and the staged method after $t=100$ iterations, as a function of the misalignment angle $\theta_{\valtrue,\querytrue}$. Columns show the match-fraction error $1-f$ (left), the value angle $\theta_{\valmap,\valtrue}$ (middle), and the query angle $\theta_{\hat{\query}_a,\querytrue}$ (right). The top row shows the mean over the ensemble of random initial assignments (shaded bands: $\pm$ one standard error); the bottom row shows the minimum over the same ensemble. Parameter values: $m=1000$, $n=15$, $d=10$, $400$ replicates per value of $\theta_{\valtrue,\querytrue}$.}
\label{fig:4}
\end{figure}

\subsection{A comparison of ${\rm EM}_{\kappa}$ and $\widetilde{\rm EM}_1$}
Figure~\ref{fig:2} presents numerical experiments where  ${\rm EM}_{0}$ and $\widetilde{\rm EM}_1$ are run  for a range of $n$ values with two different angles $\theta_{\querytrue,\valtrue}$. We have plotted the match fraction $f^{(t)} = 1- d_H(\assignv^{(t)},\assigntruev)/m$ at $t = 100$. 

From Figure \ref{fig:2}, we observe the following:  First, for both angles considered, ${\rm EM}_{0}$ achieves near-zero error for sufficiently small $n$, but there appears to be a threshold above which the algorithm does not converge for the sampled assignment. 
Second, the $\widetilde{\rm EM}_1$ algorithm converges for nearly all $n$ at the zero angle, but not $\pi/4$.
The failure at higher angles is to be expected since it violates the underlying assumption of aligned query and value. Interestingly, the match fraction in the trajectories of $\widetilde{\rm EM}_1$ also converge, but not to $f=1$.
 Intuitively, when the query and value are aligned, $\widetilde{\rm EM}_1$ has an advantage coming from the fact that the assignment and value step (E and M) use different pieces of information about the data generating process; namely, the labeling and selection rule.

\subsection{Dependence on $\kappa$ schedule}
In addition to the $\widetilde{\rm EM}_1$ and EM$_0$ simulations we performed experiments with EM$_{0.5}$  and an alternating schedule denoted $\kappa=0\to1$ or EM$_{0\to 1}$ where the assignment step alternates between $\assignextv^0$ and $\assignextv^1$.  We show simulations sampled from random initial assignments for the different methods in Fig \ref{fig:3}. 
 
  The idea of the $\kappa=0\to1$ schedule is to obtain some benefits  of $\widetilde{\rm EM}_1$, but in a method which allows for distinct $\query$ and $\val$ vectors.
 An obvious issue is that the query step is only performed after an assignment update based on the value. The query update $\hat{\query}_a$ is, however, leveraging correlations in the assigned instances at the current setup which need to have been introduced in prior steps, thus this method ends up performing only marginally better than $\widetilde{\rm EM}_1$.
 
From Figure \ref{fig:3} we see that there is a tradeoff: EM$_{0}$ schedule is able to retrieve the correct assignment from a random assignment with small probability, the mixed schedules EM$_{0.5}$ and EM$_{0\to 1}$ are more reliable, but sensitive to the angle, much like $\widetilde{\rm EM}_1$. $\widetilde{\rm EM}_1$ and EM$_{0\to 1}$ converge quickly, but often to the wrong value, and for fixed $\kappa$ the basin of attraction appears very small. Motivated by this, we tested a staged method where the first $20$ steps are run with EM$_{0\to 1}$.
 In Figure \ref{fig:4} we show the mean final match fraction after $t=100$ iterations (top row) and the minimum over the ensemble of instances for all methods, including the staged one. The staged method outperforms all others except at low angles. A detailed study of the various schedules and addressing questions such as the optimal time to switch the $\kappa$ schedule in the staged method is beyond the scope of this work.

\section{Asymptotics results for many bags}\label{sec:manybag}

Here we present results on the limit behavior of the query and value map evaluated on a random assignment. Our result for the value map tells us how many initial assignments we must sample in order for one to give a good approximation of the true assignment after one step of $\widetilde{\rm EM}_1$ when $\theta_{\valtrue,\querytrue}=0$.

\subsection{Preliminaries}

We recall the definition of a sub-Gaussian random vector and a tail bound on empirical covariance matrices with sub-Gaussian rows. 

\begin{defn} A random vector $\bx$ is sub-Gaussian if there is a constant $K$ such that
\begin{equation}
\sup_{\bu \in S^{d-1}} P\!\left(|\langle \bx,\bu \rangle| > t\right) \le 2e^{-(t/K)^2}.
\end{equation}
\end{defn}
Note that for a Gaussian $K$ is just the standard deviation. The following result is standard; e.g. see  \cite{vershynin2018high}.

\begin{lem}\label{lem:conc}
Let $X \in \mathbb{R}^{m\times d}$ have $m$ i.i.d.\ sub-Gaussian rows, the second moment matrix
\begin{equation}
\hat{\Sigma} = \frac{1}{m} X^T X \in \mathbb{R}^{d\times d}
\end{equation}
satisfies the tail bound
\begin{equation}\label{eq:conc2}
\bbP\left(\|\hat{\Sigma} - \Sigma \| >C K^2||\Sigma||\left(\sqrt{\frac{d+\delta}{m}} + \frac{d+\delta}{m} \right)\right)
< 2e^{-\delta}
\end{equation}
where $\Sigma = \E[\hat{\Sigma}]$ and $K$ is the sub-Gaussian norm of the rows of $X$ and $C$ is a constant. 
In particular, for $\epsilon \ll K^2||\Sigma||$ there is a high probability that 
\begin{equation}
\|\hat{\Sigma} - \Sigma \|= O(\epsilon)
\end{equation}
when $m \gtrsim dK^4||\Sigma||^2/\epsilon^2$. 
\end{lem}

Owing to the fact that when we condition on $d_H$ the rows of $X(\assignv)$ are distributed according to the extremal statistics of i.i.d.\ Gaussians, it is important to have some estimates of these statistics. 
Let $X_1,\dots,X_n$ be i.i.d.\ ${\rm Normal}(0,1)$ random variables and let
\[
\assigntrue=\operatorname{argmax}_{1\le j\le n} X_j.
\]
Define the following moments associated with the extremal instance:
\begin{equation}
\mu \coloneqq \mathbb{E}[X_{\assigntrue}], \qquad
s \coloneqq \mathbb{E}[X_{\assigntrue}^2], \qquad \quad w \coloneqq \mathbb{E}[X_j^2 \mid j\neq \assigntrue],\qquad 
v \coloneqq \mathbb{E}[X_{\assigntrue} X_j \mid j\neq \assigntrue], 
\end{equation}
Two important identities are the following. 
\begin{lem}\label{lem:vw}
$v = (1-s)/(n-1)$ and $w =v+1$
\end{lem}
The proof is given in Appendix \ref{app:proofs}.  Additionally, we will use the following asymptotic result, which is standard (e.g. see \cite{naess2024applied})
\begin{lem}\label{lem:extrememoments} As $n\to\infty$, $\mu \sim \sqrt{2\ln n}$  and $s \sim 2\ln n$. 
\end{lem}

\subsection{The value vector}
 If $\hat{\Sigma}(\assignv,\assignv)$ is invertible, which is true almost surely, we have 
\begin{equation}\label{eq:regxi}
\valmap(\assignv) = \hat{\Sigma}(\assignv,\assignv)^{-1}\frac{1}{m}X(\assignv)^T\by
= \hat{\Sigma}(\assignv,\assignv)^{-1}\hat{\Sigma}(\assignv,\assigntruev)\valtrue.
\end{equation} 
The idea is to approximate the matrices $\hat{\Sigma}(\assignv,\assignv)$ and $\hat{\Sigma}(\assignv,\assigntruev)$ appearing in Eq.~\eqref{eq:regxi} with their averages using Eq. \eqref{eq:conc2}.

\begin{lem}\label{lem:gammamean}
Assume Eq. \eqref{eq:datassump} and without loss of generality we take $\querytrue = \be_1$. 
Let $d_H = d_H(\assignv,\assigntruev)$ and define 
\begin{align}
\bar{A}(f)&\coloneqq \E[\hat{\Sigma}(\assignv,\assignv) \mid d_H/m= 1-f], \\
\bar{B}(f) &\coloneqq \E[\hat{\Sigma}(\assignv,\assigntruev) \mid d_H/m = 1-f]
\end{align}
where the expectations are taken over both the conditional distribution of $\assign$ and the $X^k_{i,j}$.
Then 
\begin{align}
\bar{A}(f) &= \Sigma  + \rho_n(f) \Sigma_{1,1}^{-1}\Sigma_1\Sigma_1^T\\
\bar{B}(f) &= f \Sigma +\rho_n(f) \Sigma_{1,1}^{-1}\Sigma_1\Sigma_1^T
\end{align}
where
\begin{align}
\rho_n(f) \coloneqq  f (s-1) + (1-f)(w-1)
\end{align}
\end{lem}

Note that both formulas can be written as linear interpolations between $f=0$ and $f=1$, e.g. 
\begin{align}
\bar{A}(f) = \bar{A}(0)(1-f) + \bar{A}(1)f. 
\end{align}
As a result, the proof amounts to deriving formulas for the endpoints. This is also helpful because we cannot technically apply Lemma~\eqref{lem:conc} directly to $\bar{A}(f)$ and $\bar{B}(f)$, since conditioning on $d_H$ introduces correlations among the rows. However, the rows of $\hat{\Sigma}(\assignv,\assignv)$ are conditionally independent given $1_{\{\assign_k = \assigntrue_k\}}$, and since ${\rm cov}(1_{\{\assign_k = \assigntrue_k\}},1_{\{\assign_{k'} = \assigntrue_{k'}\}}) = O(1/m)$, the conditional covariance dominates the marginal covariance asymptotically in $m$. These conditional covariance matrices have means such as $\bar{A}(1)$ and $\bar{A}(0)$.

For sufficiently large $m$, we therefore can expect $\valmap(\assignv)$ is concentrated around $\bar{A}(f)^{-1}\bar{B}(f)\valtrue$. 
To this end, we have the following result which is proved in Appendix \ref{app:proofs}. 
\begin{thm}\label{thm:conc}
Assume Eq. \eqref{eq:datassump}. Fix $f\in(0,1)$ and let $\assign$ be sampled uniformly from  
\begin{equation}
\{\assign\in[n]^m:\ d_H(\assignv,\assigntruev)/m = 1-f\}
\end{equation}
Let $\bar A,\bar B$ be as in Lemma~\ref{lem:gammamean}.
Then for $\epsilon \ll 1$, we have
\begin{equation}
\big\|\valmap(\assign) - \bar{A}(f)^{-1}\bar{B}(f)\valtrue\| = O(\epsilon)
\end{equation}
 when $m\gtrsim m_c(\epsilon)\coloneqq d (\ln n)^6/\epsilon^2$. 
\end{thm}

Since $\Sigma_{1,1}^{-1}\Sigma_1\Sigma_1^T$ in Lemma \ref{lem:gammamean} is a rank 1 perturbation, we can apply the Sherman-Morrison formula to obtain
\begin{equation}
 \bar{A}(f)^{-1} = \Sigma^{-1} - \rho_n(f) \frac{\Sigma_{1,1}^{-1}\Sigma^{-1}\Sigma_1\Sigma_1^T\Sigma^{-1} }{1+ \rho_n(f)\Sigma_{1,1}^{-1}\Sigma_1^T\Sigma^{-1}\Sigma_1 }.
\end{equation}
Since $\Sigma \querytrue = \Sigma_1$, $\querytrue = \Sigma^{-1}\Sigma_1$. Thus with $\querytrue = {\be}_1$, this simplifies to
\begin{equation}
 \bar{A}(f)^{-1} = \Sigma^{-1} - \frac{\rho_n(f)}{ \Sigma_{1,1}(1+\rho_n(f))} \querytrue (\querytrue) ^T.
\end{equation}

In order to understand what the limit vector looks like, let us consider the case $\Sigma = I$, $\querytrue = \be_1$. No generality is lost, since in practice we may estimate $\Sigma$ using all the instances and transform the instance distribution to an isotropic one. Writing $\valtrue/||\valtrue|| = u\, \querytrue + \sqrt{1-u^2}\,{\boldsymbol \delta}$ where $ {\boldsymbol \delta} \cdot \querytrue = 0$, an application of the Sherman-Morrison formula yields
\begin{equation}\label{eq:hatvconc}
\frac{1}{||\valtrue||}\bar{A}(f)^{-1}\bar{B}(f)\valtrue =\querytrue\, u\, \varphi_n(f) + f \sqrt{1-u^2}\, {\boldsymbol \delta}
\end{equation}
where 
\begin{equation}\label{eq:varphidef}
 \varphi_n(f) \coloneqq \frac{f+\rho_n(f)}{1+\rho_n(f)} = \frac{s(fn -1) + (1-f)}{n + fn (s-1) - s}.
\end{equation}
In particular, if $u =1$ (the query and value point in the same direction), the regression problem from a conditionally random assignment yields a scalar multiple of the correct query whenever $\varphi_n(f)>0$, i.e.\ for $f > (s-1)/(sn-1) \sim 1/n$. This is also evident simply from the symmetry of the problem.

\subsection{Implications for one step of the $\widetilde{\rm EM}_1$}\label{sec:onestep}
Since the assignment map $\assignextv^1$ cares only about the direction of its argument, not its magnitude, obtaining \emph{any} positive scalar multiple of $\valtrue$ from the value map is sufficient for $\widetilde{\rm EM}_1$ to recover the correct assignment. When $\query \propto \val$ (i.e.\ $u=1$), Eq.~\eqref{eq:hatvconc} becomes $\bar{A}(f)^{-1}\bar{B}(f)\valtrue = \varphi_n(f)\,\valtrue$. For a completely random assignment $\varphi_n(1/n)=1/n$. Therefore, the question becomes when $m$ is large enough that the estimation noise of Theorem~\ref{thm:conc} does not overwhelm the resulting signal, which is itself only $O(1/n)$ in size, since $\varphi_n(1/n)=1/n$.

 By Theorem~\ref{thm:conc}, $\|\valmap(\assign) - \varphi_n(f)\valtrue\| = O(\epsilon)$ once $m \gtrsim m_c(\epsilon) = d(\ln n)^6/\epsilon^2$. Requiring $\epsilon$ to be a small fraction of the signal $\varphi_n(1/n)=1/n$, say $\epsilon \sim 1/n$, gives
\begin{equation}\label{eq:mbound_onestep}
m \gtrsim d\,n^2 (\ln n)^6.
\end{equation}

It should be noted that after the initial assignment has been made, the updated assignment $\assignv^{(1)}$ is not uniformly sampled, therefore the statements above concern only the first step. Thus, it is a different and more difficult (yet feasible) problem to say something about the actual dynamics of the EM algorithm.

\subsection{The uniform selection mechanism as a special case}
The case $f=1/n$ analyzed in Section~\ref{sec:onestep} is what one obtains when the active instance is selected uniformly, i.e.\ when $\assigntruev \sim {\rm Uniform}([n]^m)$ as well as $\assignv$ -- the same as setting $\lambda=0$ (or $\kappa=0$) before taking the $\tau\to\infty$ limit, so that no extremal principle governs the labeling at all. This regime was previously analyzed by \cite{guptalearning}. A coarser version of this result, given below, is recovered from Theorem~\ref{thm:conc} and Lemma~\ref{lem:gammamean}. 

 \begin{lem}[See Lemma 3.1 in  \cite{guptalearning}]\label{thm:conc_uniform}
Assume $\bx_j^k$ are i.i.d.\ $\mathrm{Normal}(0,I_d)$, and $\assigntruev,\assignv \sim {\rm Uniform}([n]^m)$. Then, for $\epsilon \ll 1$,
\begin{equation}
||\valmap(\assign) - \valtrue/n || = O(\epsilon)
\end{equation}
 with constant probability when $m \gtrsim d/\epsilon^2$. In particular, $m \gtrsim dn^2$ samples are needed to resolve the direction of $\valtrue$ from a random assignment.
\end{lem}

We now compare to Lemma 3.1 in  \cite{guptalearning}. Their result, in our notation, concerns the large $m$ behavior of the maximizer of $Q_1(\val;p)$ when $p_{k,j} = 1/n$, i.e. $p$ is the uniform distribution on $[n]^m$.  Roughly speaking, they show that $O(dn^2\ln(n)/\epsilon^2)$ bags guarantee a high probability that the error in $\valmap(\assign)n$ is $O(\epsilon)$.  In their notation, $n$ is denoted $q$, $\epsilon$ is $\sqrt{\epsilon}$, $\valtrue$ is ${\bf r}$. Moreover, they compute the unscaled solution to the regression problem, denoted $\val_{\rm min}$ which is then multiplied by $m^{-1}\sum_{k=1}^m(X^k)^TX^k$.  To see how this is related to Lemma \ref{thm:conc_uniform}, observe that when solving the resulting regression problem, the noise floor is controlled by the estimation of the empirical covariance matrix of the assigned instances, which has $m$ rows even though we are using a total of $n \times m$ instances. Moreover, it is still a fraction $1/n$ of the total instances which provide a signal. This heuristically justifies why the scaling is the same (up to logarithmic in $n$ factors).

Compared to the extremal selection model studied in Section~\ref{sec:manybag}, this uniform-selection case is therefore not qualitatively different for a random assignment. Put simply, regardless of the selection mechanism, the signal from a random assignment is $1/n$ if $\query$ and $\val$ are aligned.

\subsection{Angle formula when query and value are unaligned}
As another consequence of Theorem \ref{thm:conc}, we have a formula for $\theta_{\valmap,\valtrue}$ in the general case when $\theta_{\valtrue,\querytrue}>0$: \begin{equation}\label{eq:thetaregb0}
\theta_{\valmap,\valtrue}  \to \bar{\theta}_{\valmap,\valtrue}(f,u)  \coloneqq \cos^{-1}\left( \frac{u^2\varphi_n(f)+ (1-u^2)f}{(u^2\varphi_n(f)^2 + (1-u^2)f^2)^{1/2}}\right),\quad u =\cos(\theta_{\valtrue,\querytrue}).
\end{equation}
Figure~\ref{fig:5} validates this formula numerically.  A noteworthy aspect of the formula is that it is non-monotonic in $f$. This also suggests that even in the unaligned case, where we no longer have a consistent estimator of the direction, the angle is still significantly  reduced from a random vector.

\begin{figure}[h!]
\centering
\includegraphics[width=\textwidth]{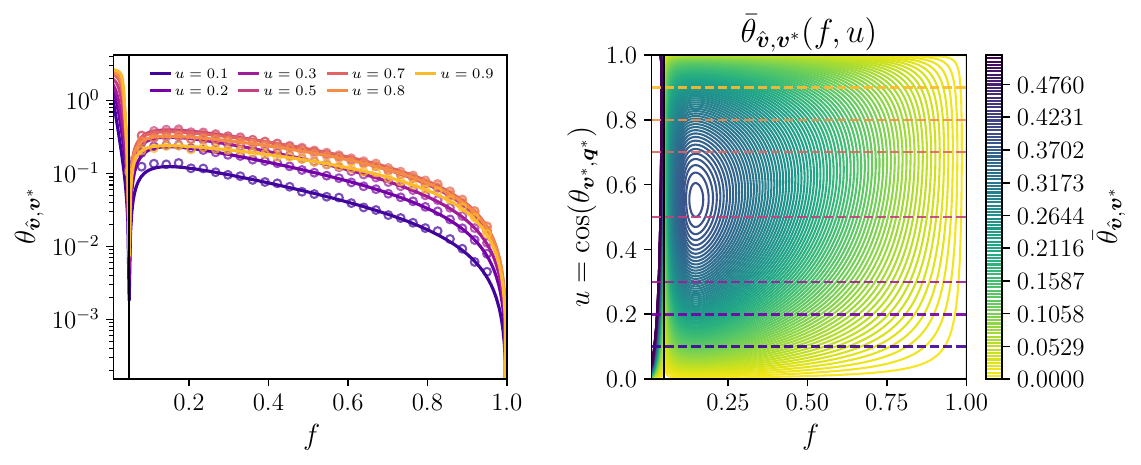}
\caption{Value update (Theorem~\ref{thm:conc}), conditioned on a given match fraction $f$. The angle $\theta_{\hat{\boldsymbol{v}},\boldsymbol{v}^*}$ is shown for several values of the alignment parameter $u=\cos(\theta_{\valtrue,\querytrue})$, (left) as simulation vs.\ theory across $f$ and (right) as a contour over $(f,u)$, with the $u$ values shown on the left indicated by dashed lines. Parameter values: $n=20$, $d=3$, $m=50{,}000$.}
\label{fig:5}
\end{figure}

\begin{figure}[h!]
\centering
\includegraphics[width=\textwidth]{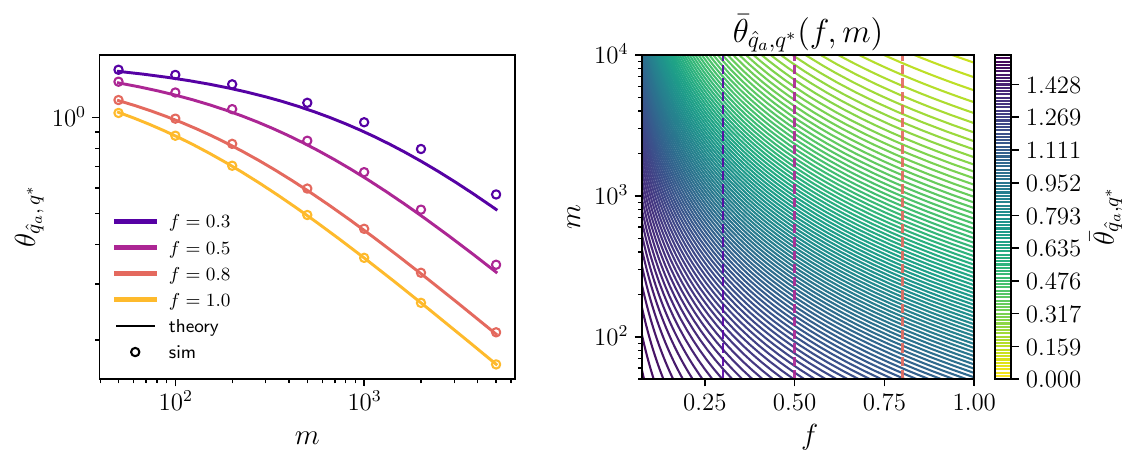}
\caption{Query update (Theorem~\ref{thm:query_asymp}). (left) the angle $\theta_{\hat{q}_a,q^*}$ as a function of $m$ for several match fractions $f$, comparing simulation to the theoretical prediction Eq.~\eqref{eq:theta_qa_theory}. (right) a contour plot of the same prediction over $(f,m)$, with the $f$ values shown on the left indicated by dashed lines. Parameter values: $n=20$, $d=500$.}
\label{fig:6}
\end{figure}

\subsection{Query update}\label{sec:query_asymp}
We end this section by presenting a result for the large $m$ behavior of the query. As with the formula above, the stated result is for isotropic instances ($\Sigma = I$) which makes the formulas simpler. Because the query update is not informative when computed from a uniformly random assignment, this result is not as useful as Theorem \ref{thm:conc}, but it does provide a rationale for using $\hat{\query}_a$ as an approximation of query MLE.

\begin{thm}\label{thm:query_asymp}
Let $\bx_j^k$ be i.i.d.\ $\mathrm{Normal}(0,I_d)$, and let $\assignv$ be sampled uniformly from
\[
\{\assignv\in[n]^m:\ d_H(\assignv,\assigntruev)/m = 1-f\}. 
\]
For any fixed $d$ and $n$, as $m\to\infty$,
\begin{equation}\label{eq:xbar_conc}
\bar{\bx}_{\assignv} \rightarrow \frac{fn-1}{n-1}\,\mu  \querytrue,\quad \text{almost surely}
\end{equation}
When $f>1/n$ the limit is a positive multiple of $\querytrue$, and consequently $\hat{\query}_a(\assignv)\to\querytrue$ almost surely.  
\end{thm}

The proof, given in Appendix \ref{app:proofs}, is based on the following idea. For a vector $\bz$ such that $\bz \cdot \querytrue =0$, $\bz \cdot \bar{\bx}_{\assignv}$ has mean $0$ and variance $O(1/m)$, while $\bar{\bx}_{\assignv} \cdot \querytrue \sim f\mu$ for large $n$ and $f\gg 1/n$. The additional $n$ dependence in Eq. \eqref{eq:xbar_conc} comes from accounting for the $1-f$ non-extremal bags.   

We can use Theorem \ref{thm:query_asymp} to obtain heuristics for the finite bag angle. Since there are $d-1$ orthogonal directions, the finite $m$ angle bias is 
\begin{equation}\label{eq:theta_qa_theory}
\theta_{\query,\querytrue} \to \bar{\theta}_{\hat{\query}_a,\querytrue}(f)
\sim \cos^{-1}\!\left(\frac{f\mu}{\sqrt{(f\mu)^2+(d-1)/m}}\right).
\end{equation}
Figure~\ref{fig:6} compares this prediction to simulation across a range of $m$ values and match fractions $f$, confirming the $1/\sqrt{m}$ convergence rate and the $f$-dependence of the signal.

\section{Conclusion}\label{sec:conclusion}

We studied an extremal multiple-instance regression problem in which the selection rule and the labeling rule are decoupled through separate query and value directions. Starting from a soft latent-variable model, we derived a family of approximate EM-based assignment iterations indexed by $\kappa \in [0,1]$, interpolating between residual-based and extremal assignment rules. The data itself is independent of  $\kappa$ in the limit when the selection and labeling are noiseless, which raises questions about how $\kappa$ should be selected. We have not given an optimal schedule for $\kappa$, thus leaving open the question of what this is even for isotropic Gaussian instances, let alone other instance distributions. 

We presented numerical results which illustrate that the optimal schedule is highly angle-dependent. The alternating schedule performs better for large misalignment between the value and query directions, while the reverse is true for the fixed-but-mixed schedule. This has to do with how the M-step biases the assignment. In the alternating schedule, the bias introduced by one $\kappa$ value can interfere with the computation of the value vector in the next step when the query and value are misaligned.

The main theoretical results describe the large-$m$ behavior of the value and query update maps conditional on a fixed match fraction $f$. Theorem~\ref{thm:conc} shows that the value update $\valmap(\assign)$ concentrates around a deterministic function of $f$ and the alignment $u=\cos(\theta_{\valtrue,\querytrue})$, with error $O(\epsilon)$ once $m\gtrsim m_c(\epsilon)\coloneqq d(\ln n)^6/\epsilon^2$. As a consequence, we find that a single random initial assignment already yields, after one step, an assignment close to the truth once $m\gtrsim dn^2(\ln n)^6$ bags are available (Section~\ref{sec:onestep}).  Taken together, these results establish consistency of the single-step update maps at fixed $f$ in the large-$m$ limit, and show that one-step recovery from a fully random initialization requires only polynomially more bags than recovery from a favorable one. In either case, the results do not by themselves characterize the full finite-time dynamics of the iteration, which is left for future work.

In addition to motivating problems at the interface of high dimensional probability and extremal statistics, it is our hope that the ideas will lead to improved methods for practical problems. In particular, it would be interesting to investigate the implications for real MIL pipelines. As a step towards this, one could consider the situation where the instances are learned embeddings of samples from a sequence space. It would also be interesting to study self-attention. 

\section*{Acknowledgements}

Early work on this project was conducted during the workshop \textit{Extremal Statistics in Biology} at the Erwin Schr\"{o}dinger International Institute for Mathematics and Physics (ESI), Vienna, June 2025. This research was supported in part by grants from the NSF (DMS-2235451) and Simons Foundation (MPS-NITMB-00005320) to the NSF-Simons National Institute for Theory and Mathematics in Biology (NITMB). I also acknowledge helpful conversations with Pankaj Mehta about this project. 

\newpage 
\appendix
\numberwithin{equation}{section}
\makeatletter
\@addtoreset{figure}{section}
\makeatother
\renewcommand{\thefigure}{\thesection.\arabic{figure}}
\renewcommand{\theHfigure}{\thesection.\arabic{figure}}
\setcounter{figure}{0}

\section{Notation}\label{app:notation}

Table~\ref{table:1} summarizes the notation used in this manuscript. 

\begin{table}[h!]
\caption{Notation used in the main text.}
\label{table:1}
\centering
\begin{tabular}{@{}l p{0.64\textwidth} l@{}}
\toprule
\multicolumn{3}{l}{\textit{Data and parameters}} \\
\midrule
$m,\,n,\,d$ & Number of bags; instances per bag; feature dimension & \\
$\theta_{\bx,\by}$ & Angle between vectors $\bx$ and $\by$ &  \\
$\bx_j^k \in \reals^d$ & Instance $j$ in bag $k$, $1\le j\le n$, $1\le k\le m$ &  \\
$y_k \in \reals$ & Bag-level response &  \\
$\query,\,\val$ & query and value vectors &   \\
$\kappa \in [0,1]$ & Interpolation parameter. &  \\
$\assignv \in [n]^m$ & Assignment vector $(\assign_1,\dots,\assign_m)$ &  \\
$d_H(\assignv,\assigntruev)$ & Hamming distance $\#\{k:\assign_k\ne\assigntrue_k\}$ &  \\
$X(\assignv)$ & Design matrix for assignment $\assignv$ &  \\
$\hat{\Sigma}(\assignv,\assignv')$ & Cross covariance matrix $m^{-1}X(\assignv)^TX(\assignv')$  &  \\
$f = 1 - d_H(\assignv,\assigntruev)/m$ & Match fraction & \\[4pt]
\multicolumn{3}{l}{} \\
\multicolumn{2}{l}{\textit{Update maps for ${\rm EM}_{\kappa}$ iteration}} & Domain $\to$ Range \\
\midrule
$\assignext_k^\kappa(\query,\val)$ & Interpolating assignment in bag $k$ (Eq.~\eqref{eq:assignkappa}) & $\reals^d\times\reals^d \to [n]$ \\
$\assignextv^\kappa(\query,\val)$ & Bag-wise version $(\assignext_1^\kappa,\dots,\assignext_m^\kappa)$ & $\reals^d\times\reals^d \to [n]^m$ \\
$\valmap(\assignv)$ & Value regression map (Eq.~\eqref{eq:regxi}) & $[n]^m \to \reals^d$ \\
$\hat{\query}_a(\assignv)$ & average estimate of query & $[n]^m \to S^{d-1}$ \\[4pt]
\multicolumn{3}{l}{} \\
\multicolumn{2}{l}{\textit{Extremal moments and model parameters (\S2, \S4)}} & \\
\midrule
$\mu,\,s$ & Mean, second moment of the extremal instance: $\mu\coloneqq\E[X_{\assigntrue}]$, $s\coloneqq\E[X_{\assigntrue}^2]$ &  \\
$w,\,v$ & Second moment, cross-moment of a non-extremal instance: $w\coloneqq\E[X_j^2\mid j\ne\assigntrue]$, $v\coloneqq\E[X_{\assigntrue}X_j\mid j\ne\assigntrue]$ &  \\
$\rho_n(f)$ & Interpolation coefficient, $\rho_n(f)\coloneqq f(s-1)+(1-f)(w-1)$ (Lemma~\ref{lem:gammamean}) &  \\
$\varphi_n(f)$ & Value-map scaling factor (Eq.~\eqref{eq:varphidef}) &  \\
$\sigma_\epsilon,\,\lambda,\,\tau$ & Labeling-noise s.d.; selection inverse temperature; $\tau\coloneqq\lambda+1/(2\sigma_\epsilon^2)$ &  \\
$\vartheta = (\val,\query)$ & Model parameter pair &  \\
\bottomrule
\end{tabular}
\end{table}

\section{Proofs}\label{app:proofs}

We first prove Lemma \ref{lem:vw}, which gives the identities for the extremal moments.

\begin{proof}
We first show  $\E\left[X_{\assigntrue}X_1\right]=1/n$. Note that, using standard Gaussian tail integrals
\begin{align}
\E[X_{\assigntrue}X_1|\{\max_{j >1}X_j =m\}] &= \E[X_1^2{\bf 1}\{X_1>m\}]  + m\E[X_1{\bf 1}\{X_1 < m\}]\\
&= m \phi(m) + 1-\Phi(m) - m\phi(m) = 1- \Phi(m)
\end{align}
where $\phi$ and $\Phi$ are the pdf and cdf of a standard normal. Averaging over $m$ gives
\begin{equation}
\E[1-\Phi(m)] = \bbP(X_1> \max_{j >1}X_j) = \frac{1}{n}. 
\end{equation}
Hence 
\begin{equation}
\E[X_{\assigntrue}\sum_{j=1}^nX_j]  = 1 = s + (n-1)v,
\end{equation}
where the last equality holds by exchangeability. The formula for $v$ follows.

For $w$, since $\E[X_j^2]=1$,
\begin{equation}
n = \E\Big[\sum_{j=1}^n X_j^2\Big] = s + (n-1)w,
\end{equation}
decomposing the sum over $\{j = \assigntrue\}$ and $\{j \neq \assigntrue\}$ exactly as above. Hence $w = (n-s)/(n-1)$.
\end{proof}

Turning to our results about the covariance matrices, we introduce the notation 
 \begin{align}
 A &\coloneqq \hat{\Sigma}(\assignv,\assignv) \quad B \coloneqq \hat{\Sigma}(\assignv,\assigntruev)\\
 \end{align}
 and keep $\bar A,\bar B$ as defined in Lemma~\ref{lem:gammamean}. 

\begin{proof}[Proof of Lemma \ref{lem:gammamean}]
If $k$ is a matched instance ($\assign_k = \assigntrue_k$), the elements of the covariance matrix satisfy
\begin{align}
\E[X_{\assigntrue_k,j'}^k X_{\assigntrue_k,j}^k]
&= \E\!\bigl[\E[X_{\assigntrue_k,j'}^k X_{\assigntrue_k,j}^k \mid X_{\assigntrue_k,1}^k]\bigr] \\
&= \E\!\bigl[\mathrm{Cov}(X_{\assigntrue_k,j'}^k, X_{\assigntrue_k,j}^k \mid X_{\assigntrue_k,1}^k)
  + \E[X_{\assigntrue_k,j}^k \mid X_{\assigntrue_k,1}^k]\E[X_{\assigntrue_k,j'}^k \mid X_{\assigntrue_k,1}^k]\bigr] \\
&= \Sigma_{j,j'} - \Sigma_{1,1}^{-1}\Sigma_{j,1}\Sigma_{j',1} + s\,\Sigma_{1,1}^{-1}\Sigma_{j,1}\Sigma_{j',1}.
\end{align}
If $k$ is an unmatched instance ($\assign_k \neq \assigntrue_k$), then $X_{\assigntrue_k,j}^k$ and $X_{\assign_k,j}^k$ are correlated through their shared dependence on the first coordinate. In particular, for the cross term,
\begin{align}
\E[X_{\assigntrue_k,j'}^k X_{\assign_k,j}^k]
&= \E\!\bigl[\E[X_{\assigntrue_k,j'}^k \mid X_{\assigntrue_k,1}^k]\,\E[X_{\assign_k,j}^k \mid X_{\assign_k,1}^k]\bigr] \\
&= \Sigma_{1,1}^{-2}\Sigma_{1,j'}\Sigma_{1,j}\,\E[X_{\assigntrue_k,1}^k X_{\assign_k,1}^k]
 = \Sigma_{1,1}^{-1}\Sigma_{1,j'}\Sigma_{1,j}\,v,
\end{align}
and for an unmatched instance with itself,
\begin{align}
\E[X_{\assign_k,j'}^k X_{\assign_k,j}^k]
&= \E\!\bigl[\mathrm{Cov}(X_{\assign_k,j}^k, X_{\assign_k,j'}^k \mid X_{\assign_k,1}^k)
  + \E[X_{\assign_k,j'}^k \mid X_{\assign_k,1}^k]\E[X_{\assign_k,j}^k \mid X_{\assign_k,1}^k]\bigr] \\
&= \Sigma_{j,j'} - \Sigma_{1,1}^{-1}\Sigma_{j,1}\Sigma_{j',1} + w\,\Sigma_{1,1}^{-1}\Sigma_{j,1}\Sigma_{j',1},
\end{align}
Using the identity $\E[\hat{\Sigma}(\assignv,\assignv) \mid f] = f\,\E[\hat{\Sigma}(\assigntruev,\assigntruev)] + (1-f)\,\E[\hat{\Sigma}(\assignv,\assignv) \mid d_H = m]$ and collecting the matched and unmatched contributions,
\begin{align}
\E[\hat{\Sigma}(\assigntruev,\assigntruev)] &= \Sigma + (s-1)\,\Sigma_{1,1}^{-1}\Sigma_1\Sigma_1^T, \\
\E[\hat{\Sigma}(\assignv,\assignv) \mid d_H = m] &= \Sigma +(w-1)\,\Sigma_{1,1}^{-1}\Sigma_1\Sigma_1^T,
\end{align}
we obtain 
\begin{equation}
\bar{A}(f) = \Sigma + \bigl(f(s-1) + (1-f)(w-1)\bigr)\Sigma_{1,1}^{-1}\Sigma_1\Sigma_1^T = \Sigma + \rho_n(f)\,\Sigma_{1,1}^{-1}\Sigma_1\Sigma_1^T,
\end{equation}
with $\rho_n(f) = f(s-1)+(1-f)(w-1)$.  Similarly,
\begin{align}
\E[\hat{\Sigma}(\assignv,\assigntruev) \mid d_H = m] &= v\,\Sigma_{1,1}^{-1}\Sigma_1\Sigma_1^T,
\end{align}
and therefore
\begin{equation}
\bar{B}(f) = f\,\Sigma + \bigl(f(s-1) + (1-f)v\bigr)\Sigma_{1,1}^{-1}\Sigma_1\Sigma_1^T = f\Sigma + \rho_n(f)\,\Sigma_{1,1}^{-1}\Sigma_1\Sigma_1^T,
\end{equation}
where we have used $w = v+1$. 
\end{proof}

The proof of Theorem \ref{thm:conc} will make use of the following fact, which follows from Lemma \ref{lem:extrememoments}

\begin{proof}[Proof of Theorem \ref{thm:conc}]
Write
\begin{align}
A^{-1}B
&= \bar{A}^{-1}\bar{B} + (A^{-1} - \bar{A}^{-1})\bar{B} + A^{-1}(B-\bar{B}) = \bar{A}^{-1}\bar{B} + u,
\end{align}
with
\begin{equation}
u=\bar{A}^{-1}(A-\bar{A})A^{-1}\bar{B}+A^{-1}(B-\bar{B}).
\end{equation}
For $\eta>0$, let
\begin{equation}
E_\eta\coloneqq\{\|A-\bar A\|\le \eta,\ \|B-\bar B\|\le \eta\}.
\end{equation}
Choose $\eta\le (2\|\bar A^{-1}\|)^{-1}$. On $E_\eta$,
\begin{equation}
\|\bar A^{-1}(A-\bar A)\|\le \tfrac12,
\end{equation}
so $A$ is invertible and $\|A^{-1}\|\le 2\|\bar A^{-1}\|$. Hence on $E_\eta$,
\begin{equation}
\|u\|
\le (2\|\bar A^{-1}\|^2\|\bar B\|+ 2\|\bar A^{-1}\|)\eta
\end{equation}
Let
\begin{equation}
c_n \coloneqq 2\|\bar A(f)^{-1}\|^2\|\bar B(f)\|+ 2\|\bar A(f)^{-1}\|
\end{equation}
Note $c_n = O(s)$ in $n$ (note however it also depends on $f$ which is bounded by $1$).

If $\eta\coloneqq \epsilon/(2c_n)$ then
\begin{equation}
\bbP\!\left(\|\valmap(\assign)-\bar A^{-1}\bar B\,\valtrue\|>\epsilon\right)
\le \bbP(E_\eta^c).
\end{equation}
By the union bound,
\begin{equation}
\bbP(E_\eta^c)
\le\bbP(\|A-\bar A\|>\eta)+\bbP(\|B-\bar B\|>\eta).
\end{equation}
We now need to show these two terms vanish as $m \to \infty$. We begin with the $A$ term.

As noted in the main text, we need to address the fact that the rows of $A$ are not independent, and thus standard concentration results do not directly apply. Let
\begin{equation}
\iota_k\coloneqq 1_{\{\assign_k=\assigntrue_k\}}
\end{equation}
 and $U^k\coloneqq\bx_{\assign_k}^k(\bx_{\assign_k}^k)^T$. We order the bags so that the first $\lfloor fm \rfloor$ are the matched ones and write
 \begin{align}
 A&=\frac{1}{m}\sum_k U^k  = f\frac{1}{fm} \sum_{k=1}^{fm}U^k 
 +   (1-f)\frac{1}{(1-f)m } \sum_{k=\lfloor fm \rfloor + 1}^{m}U^k\\
 & \coloneqq fA^{(1)} + (1-f)A^{(0)}
 \end{align}
 (where we have dropped the floor symbols since they are irrelevant in the limit).
 Because distinct bags are independent, the $U^k$ are conditionally independent given $\{\iota_i\}_{i=1}^m$.
Then 
\begin{equation}
\bar{A}(1) = \E[A^{(1)}\mid\iota_k=1],\quad \bar{A}(0)\coloneqq\E[A^{(0)}\mid\iota_k=0].
\end{equation}
Using the triangle inequality 
\begin{align}
\bbP(\|A-\bar A\|>\eta) &\le \bbP(f\|A^{(1)}-\bar{A}(1) \|+ (1-f)\|A^{(0)}-\bar{A}(0)\|  >\eta)  \\
  &\le \bbP(\|A^{(1)}-\bar{A}(1)\|>\eta/(f2)) + \bbP(\|A^{(0)}-\bar{A}(0)\|>\eta/((1-f)2))
\end{align}
The concentration result for sub-Gaussian random matrices (Eq. \eqref{eq:conc2}) applies to both terms. Recall that the sub-Gaussian norm of $X^k_{\assigntrue_k}$ is $O(\sqrt{s})$. Also 
\begin{equation}
||A^{(1)}|| = O(s),\quad ||A^{(0)}|| = O(1)
\end{equation}
 by Lemmas \ref{lem:extrememoments} and \ref{lem:gammamean}.

Since the sub-Gaussian norm of $X_{\assigntrue_k}^k$ is $O(\sqrt{s})$ and $||A(1)|| = O(s)$, Lemma \ref{lem:conc} implies that for $\eta \ll 2fs^3$,  $\|A^{(1)}-\bar{A}(1)\|>\eta/(2f)$ with constant probability 
when
\begin{equation}
mf > \frac{4ds^4f^2}{\eta^2}
= \frac{16ds^4f^2c_n^2}{\epsilon^2}
= O\left(\frac{d(\ln n)^6f^2}{\epsilon^2}\right)
\end{equation}
Similarly, when $\eta \ll 2(1-f)$, $\|A^{(0)}-\bar{A}(0)\|>\eta/(2(1-f))$ with constant probability 
if 
\begin{equation}
m(1-f)  \gtrsim \frac{d4(1-f)^2}{\eta^2} = \frac{d4(1-f)^2c_n^2}{\epsilon^2}  = O\left(\frac{d (\ln n)^2(1-f)}{\epsilon^2}\right).
\end{equation}
Therefore when $\epsilon \ll 2\ln n \min\{f,1-f\}$ , 
\begin{equation}
m \gtrsim  \frac{d (\ln n)^2((\ln n)^4f + 1-f)}{\epsilon^2}  = O\left( \frac{d (\ln n)^6}{\epsilon^2}\right). 
\end{equation}

We now need to show that the same is true for the $B$ term. Since $B$ is a cross-covariance matrix, we embed it into a $2d\times 2d$ covariance matrix
and apply the same argument. Define $z_k\in\reals^{2d}$ by
\begin{equation}
z_k\coloneqq\begin{bmatrix}\bx_{\assign_k}^k\\ \bx_{\assigntrue_k}^k\end{bmatrix},\qquad
L^k\coloneqq z_kz_k^T,\qquad
\Gamma\coloneqq \frac1m\sum_{k=1}^m L^k
=\begin{bmatrix}
A & B\\
B^T & \hat\Sigma(\assigntruev,\assigntruev)
\end{bmatrix}.
\end{equation}
The argument used for $A$ can be applied to $\Gamma$, which now has $(2d)^2$ entries, but the same dominant concentration behavior. The result therefore follows from Eq. \eqref{eq:conc2}. 
Let
\begin{equation}
\bar{\Gamma}(f)\coloneqq\E[\Gamma\mid f].
\end{equation}
Indeed, since $B-\bar B$ is the upper-right block of $\Gamma-\bar\Gamma(f)$, we have
\begin{equation}
\|B-\bar B\|
\le
\left\|\Gamma-\bar\Gamma(f)\right\|.
\end{equation}
Thus it is enough to bound the right hand side. Splitting the bags according to whether they are matched gives
\begin{equation}
\Gamma
= f\Gamma^{(1)}+(1-f)\Gamma^{(0)},
\qquad
\Gamma^{(1)}\coloneqq \frac{1}{fm}\sum_{\iota_k=1} L^k,
\quad
\Gamma^{(0)}\coloneqq \frac{1}{(1-f)m}\sum_{\iota_k=0} L^k,
\end{equation}
where, as above, the floor symbols are irrelevant in the limit. The corresponding expectations are
\begin{equation}
\bar{\Gamma}(1)=\E[\Gamma^{(1)}\mid\iota_k=1],
\quad
\bar{\Gamma}(0)=\E[\Gamma^{(0)}\mid\iota_k=0],
\quad
\bar{\Gamma}(f)=f\bar{\Gamma}(1)+(1-f)\bar{\Gamma}(0).
\end{equation}
Using the triangle inequality,
\begin{equation}
\bbP\!\left(\left\|\Gamma-\bar\Gamma(f)\right\|>\eta\right)
\le
\bbP\!\left(
\left\|
\Gamma^{(1)}-\bar{\Gamma}(1)
\right\|>\eta/(2f)
\right)
+
\bbP\!\left(
\left\|
\Gamma^{(0)}-\bar{\Gamma}(0)
\right\|>\eta/(2(1-f))
\right).
\end{equation}
Conditioned on the matched set, the $z_k$'s are independent in each of the two sums. In the matched case $z_k$ consists of two copies of the extremal instance, while in the unmatched case one component is extremal and the other is non-extremal. Hence the sub-Gaussian norm and covariance norm are bounded by the same dominant extremal scale used above for $A$, up to constants.
Consequently $\bbP(\|B-\bar B\|>\eta)$ too has the same dominant bound as the $A$ term, up to replacing $d$ by $2d$. 
\end{proof}

\begin{proof}[Proof of Theorem \ref{thm:query_asymp}]
We make the decomposition 
 \begin{align}
 \bx_j^k=s_{k,j} \querytrue+\bz_j^k
 \end{align}
  The true assignment $\assigntrue_k=\operatorname{argmax}_j\bx_j^k\cdot\querytrue = \operatorname{argmax}_j s_{k,j}$ depends only on $\{s_{k,j}\}_{j=1}^n$.
  We also have 
  \begin{equation}
  \E[ \bx_j^k|j = \assigntrue] = \E[s_{k,\assigntrue_k}]\querytrue = \mu \querytrue
  \end{equation} and because
  \begin{equation}
  \E[ \bx_j^k] =   \E[ \bx_j^k|j = \assigntrue]\frac{1}{n} +  \E[ \bx_j^k|j \ne \assigntrue] \frac{n-1}{n}
  \end{equation}
  we get 
  \begin{equation}
\E[ \bx_j^k|j \ne \assigntrue]  = - \frac{\querytrue \mu}{n-1}.
  \end{equation}
  Therefore, for fixed $n$, by the law of large numbers yields 
  \begin{equation}
  \bar{\bx}_{\assignv} \xrightarrow{{\rm a.s.}} f \mu \querytrue - \frac{(1-f)\mu}{n-1} \querytrue = \frac{fn-1}{n-1}\,\mu\querytrue,
  \end{equation}
  establishing~\eqref{eq:xbar_conc}. When $f>1/n$, the limit $(fn-1)/(n-1)\mu>0$, so $\|\bar{\bx}_{\assignv}\|\to (fn-1)/(n-1)\mu$ and $\hat{\query}_a = \bar{\bx}_{\assignv}/\|\bar{\bx}_{\assignv}\|\to\querytrue$ almost surely. When $f=1/n$ the limit is zero.
\end{proof}

\bibliographystyle{siam}
\bibliography{./extremal_regression.bib}

\end{document}